\documentclass[12pt]{amsart}

\newif\ifps
\pstrue

\usepackage{amsmath, amsfonts, latexsym, array, amssymb}
\usepackage{enumitem}
\usepackage{psfrag}
\usepackage[all]{xypic}
\usepackage[all]{xy}
\usepackage{graphicx}
\usepackage[svgnames]{xcolor}
\usepackage{epstopdf}

\newtheorem{thm}{Theorem}[section]
\newtheorem{cor}[thm]{Corollary}
\newtheorem{lem}[thm]{Lemma}

\newtheorem{defn}[thm]{Definition}

\newtheorem{thma}{Theorem}

\def\epsilon{\varepsilon}

\newcommand{\occult}[1]{}

\newcommand{\old}[1]{}

\usepackage{ulem}

\newcommand{\htop}{h_\mathrm{top}}
\newcommand{\eps}{\varepsilon}
\newcommand{\ph}{\varphi}
\newcommand\diam{{\operatorname{diam}}}

\newcommand\RR{{\mathbb R}}

\newcommand\TT{{\mathbb T}}
\newcommand\ZZ{{\mathbb Z}}
\newcommand\NN{{\mathbb N}}

\newcommand{\second}{r}

\newcommand{\Lspan}{\Lambda^\mathrm{span}}
\newcommand{\Lsep}{\Lambda^\mathrm{sep}}

\newcommand{\CCC}{\mathcal{C}}
\newcommand{\GGG}{\mathcal{G}}

\newcommand{\DDD}{\mathcal{D}}
\newcommand{\FFF}{\mathcal{F}}
\newcommand{\SSS}{\mathcal{S}}
\newcommand{\PPP}{\mathcal{P}}

\newcommand{\UUU}{\mathcal{U}}

\newcommand{\Pexp}{P_\mathrm{exp}^\perp}

\newcommand{\Pbad}{\Psi}
\newcommand{\bk}{\bar\kappa}
\newcommand{\phigeo}{\varphi^u}
\DeclareMathOperator{\Diff}{Diff}
\DeclareMathOperator{\Var}{Var}

\numberwithin{equation}{section}

\begin{document}

\title{Equilibrium states for Ma\~n\'e diffeomorphisms}

\author{Vaughn Climenhaga}
\author{Todd Fisher}
\author{Daniel J. Thompson}

\address{Department of Mathematics, University of Houston, Houston, Texas 77204.}
\email{climenha@math.uh.edu}

\address{Department of Mathematics, Brigham Young University, Provo, UT 84602}
\email{tfisher@mathematics.byu.edu}

\address{Department of Mathematics, The Ohio State University, 100 Math Tower, 231 West 18th Avenue, Columbus, Ohio 43210.}
\email{thompson@math.osu.edu}

\date{\today}


\thanks{V.C.\ is supported by NSF grants DMS-1362838 and 
DMS-1554794.  T.F.\ is supported by Simons Foundation grant \# 239708. D.T.\ is supported by NSF grant DMS-$1461163$}


\begin{abstract}
We study thermodynamic formalism for the family of robustly transitive diffeomorphisms introduced by Ma\~n\'e, establishing existence and uniqueness of equilibrium states for natural classes of potential functions. In particular, 
we characterize the SRB measures for these  diffeomorphisms as unique equilibrium states for a suitable geometric potential.  We also obtain large deviations and multifractal results for the unique equilibrium states produced by the main theorem.
 \end{abstract}

\maketitle
\normalem

\section{Introduction}

Let $f\colon M\to M$ be a diffeomorphism of a compact smooth manifold.   Among the invariant probability measures for the system, thermodynamic formalism identifies distinguished measures called \emph{equilibrium states}; these are measures that maximize the quantity $h_\mu(f) + \int \ph\,d\mu$, where $\ph\colon M\to \RR$ is a \emph{potential function}.

Sinai, Ruelle, and Bowen \cite{jS72,rB75,dR76} showed that a mixing Anosov diffeomorphism has a unique equilibrium state $\mu_\ph$ for every H\"older continuous potential $\ph$, and that for the \emph{geometric potential} $\phigeo(x) = -\log|\det Df|_{E^u(x)}|$, this unique equilibrium state  is the \emph{SRB measure}, which is the physically relevant invariant measure. The extension of this theory to systems beyond uniform hyperbolicity has generated a great deal of activity \cite{DU,lY98,oS01,BT09,ST16,PSZ}. In this paper, we study a class of derived from Anosov (DA) partially hyperbolic diffeomorphisms using theory developed by the first and third authors in \cite{CT4}. The results from \cite{CT4} show that equilibrium states exist and are unique under the hypotheses that  `obstructions to the specification property and regularity' and `obstructions to expansivity' carry less topological pressure than the whole space.  

We consider the class of diffeomorphisms introduced by Ma\~n\'e~\cite{man78}, which are partially hyperbolic maps $f_M\colon \TT^d\to \TT^d$ ($d\geq 3$) constructed as $C^0$-perturbations of a hyperbolic toral automorphism $f_A$ with 1-dimensional unstable bundle.   
These maps are robustly transitive but not Anosov; they admit an invariant splitting $T\TT^d = E^u \oplus E^c \oplus E^s$ where vectors in $E^c$ are sometimes expanded and sometimes contracted.  We give explicit criteria under which $f_M$, or a $C^1$-perturbation, has a unique equilibrium state for a H\"older continuous potential $\ph\colon \TT^d\to\RR$.  

In \cite{CFT_BV}, we gave analogous results for the Bonatti--Viana family of diffeomorphisms, which admit a dominated splitting but are not partially hyperbolic. 
Our results here for the Ma\~n\'e family  are stronger; in particular, the fact that the unstable bundle is 1-dimensional and uniformly expanding allows us to work at arbitrarily small scales and obtain large deviations and multifractal results not present in \cite{CFT_BV}.  Our proofs, which rely on general pressure estimates for DA diffeomorphisms from \cite{CFT_BV}, are correspondingly simpler. An additional novelty in this paper is that we apply our results to a larger class of H\"older potential functions by giving a criteria for uniqueness involving the H\"older semi-norm. We emphasize that although we choose to focus on the Ma\~n\'e class since it is a model class of partially hyperbolic examples in the literature, our approach to existence and uniqueness applies more generally and does not rely on 1-dimensionality of the unstable manifold in an essential way, as is made clear in \cite{CFT_BV}.

We need to control two parameters $\rho,\second>0$ in Ma\~n\'e's construction. We write $\FFF_{\rho,\second}$ for the set of Ma\~n\'e diffeomorphisms $f_M$ such that
\begin{enumerate}[label=(\roman{*})]
\item $f_A$ has a fixed point $q$ such that $f_M = f_A$ on $\TT^d \setminus B(q,\rho)$, and
\item 
\label{gamma}
if an orbit spends a proportion at least $\second$ of its time outside $B(q,\rho)$, then it contracts vectors in $E^c$.
\end{enumerate}
The parameter $\second$ has a more intrinsic definition as an upper bound on a quantity involving the maximum derivative of $f_M$ on $E^c$, and the construction can be carried out so $r$ is arbitrarily small.
Our results apply to $C^1$ perturbations of $f_M$ that are partially hyperbolic, dynamically coherent, and satisfy \ref{gamma}. We denote this $C^1$-open set by $\UUU_{\rho,\second}$. We describe the Ma\~n\'e construction and $\UUU_{\rho,\second}$ more precisely in \S\ref{s.constructions}. 

We now state our main theorem. 
Here  $h$ is the topological entropy of $f_A$, $L$ is a constant defined in \S \ref{pressureestimate} depending on $f_A$ and the maximum of $d_{C^0}(g, f_A)$ for $g \in \UUU_{\rho,\second}$, and
$H(r) = -r \log r -(1-r) \log (1-r)$.

\begin{thma}\label{t.mane} 
Given $g\in \UUU_{\rho,\second}$ and $\varphi\colon \mathbb{T}^d\to \mathbb{R}$ H\"older continuous,
let
\[
\Psi(\rho,\second,\ph) = (1-r)\sup_{B(q,\rho)} \ph + r(\sup_{\TT^d} \ph + h + \log L) + H(2r).
\]
If $\Psi(\rho,r,\varphi)< P(\varphi; g)$, then $(\TT^d,g,\varphi)$ has a unique equilibrium state. 
\end{thma}

We then apply Theorem \ref{t.mane} by finding sufficient conditions to verify the inequality $\Psi(\rho,r,\varphi)< P(\varphi; g)$. In \S\ref{cor.br}, we show that for a fixed diffeomorphism in $\UUU_{\rho,\second}$, every H\"older potential satisfying a bounded range condition has a unique equilibrium state. 

In \S \ref{sec:corollaries}, we obtain estimates on $\Psi(\rho,r,\varphi)$ and $P(\varphi; g)$ in terms of the H\"older semi-norm $|\ph|_\alpha$ of the potential. 
We apply these estimates to obtain the following theorem, which says that the set of potentials for which Theorem \ref{t.mane} applies for $g\in \UUU_{\rho,\second}$ contains a ball around the origin in $C^\alpha(\TT^d)$ whose radius goes to $\infty$ as $\rho,\second\to0$.


\begin{thma}\label{cor1.2}
There is a function $T(\rho,\second; \alpha)$ 
with the property that 
\begin{enumerate}[label=\textup{(\arabic{*})}]
\item if $g\in \UUU_{\rho,\second}$ and $|\ph|_\alpha < T(\rho,\second; \alpha)$, then $(\TT^d,g,\ph)$ has a unique equilibrium state; and
\item $T(\rho,\second; \alpha)\to\infty$ as $\rho,\second\to0$.
\end{enumerate}
\end{thma}
As an immediate consequence, we see that for a fixed H\"older potential $\ph$, there exist $\rho, \second$ so that there is a unique equilibrium state with respect to any $g \in \UUU_{\rho,\second}$.

Theorem \ref{cor1.2} is proved using Theorem \ref{thm:pressure-gap}, which is a general lower bound on the entropy of an equilibrium state with respect to $f_A$ in terms of the H\"older norm of the potential that allows us to estimate $P(\ph;g)$ from below.

We apply Theorem \ref{t.mane} to scalar multiples of the geometric potential $\phigeo(x) = \phigeo_g(x) = -\log\|Dg|_{E^u(x)}\|$. We obtain the following results. 

\begin{thma}\label{main3} 
Let $g \in \UUU_{\rho, \second}$ be a $C^2$ diffeomorphism. Suppose that
\begin{equation}\label{eqn:mane-srb-condition}
r(h + \log L )+H(2r) < \min \left\{\frac{\sup_{x\in \TT^d} \phigeo_g(x)}{\inf_{x\in\TT^d} \phigeo_g(x)} h, -\sup_{x\in \TT^d} \phigeo_g(x)\right\}.
\end{equation}
Then the following properties hold.
\begin{enumerate}[label=\textup{(\arabic{*})}]
\item $t=1$ is the unique root of the function $t\mapsto P(t\phigeo_g)$;
\item\label{ag} There exists $a=a(g)>0$ such that $t\phigeo$ has a unique equilibrium state $\mu_t$ for each $t\in (-a,1+a)$; 
\item $\mu_1$ is the unique SRB measure for $g$.
\end{enumerate}
\end{thma}
 
The quantity $\sup \phigeo / \inf \phigeo$  is uniformly positive for all Ma\~n\'e diffeomorphisms, and the construction can be carried out so $ \sup \phigeo_{f_M}$ and $\inf \phigeo_{f_M}$ are both close to 
$ \phigeo_{f_A} = - \log \lambda_u = - h$, where $\lambda_u$ is the unique eigenvalue of $A$ greater than $1$, so the right hand side of \eqref{eqn:mane-srb-condition} is close to $\log \lambda_u$. Thus, the inequality \eqref{eqn:mane-srb-condition} holds when $\rho, r$ are small. For a  sequence of Ma\~n\'e diffeomorphisms $f_k \in \FFF_{\rho_k, \second_k}$ with $ \rho_k, \second_k\to 0$, it is easy to ensure that the H\"older semi-norm of $\phigeo_{f_k}$ is uniformly bounded (for example by ensuring that the restrictions of each $f_k$ to $B(q, \rho_k)$ are rescalings of each other). In this case, applying Theorem \ref{cor1.2}, we see that $a(f_k) \to \infty$ as $k \to \infty$ in \ref{ag} above.

In \S\ref{s.ldp}, we derive consequences of Theorem \ref{main3} for the multifractal analysis of the largest Lyapunov exponent, and give an upper large deviations principle for the equilibrium states in our main theorems. 

We now discuss related results in the literature for partially hyperbolic systems, which have largely focused on the measure of maximal entropy (MME).  For ergodic toral automorphisms, the Haar measure was shown to be the unique MME by Berg \cite{Berg} using convolutions. 
Existence of a unique MME for the Ma\~n\'e examples was obtained in~\cite{BFSV}.  For partially hyperbolic diffeomorphisms of the 3-torus homotopic to a hyperbolic automorphism, uniqueness of the MME was proved by Ures \cite{Ur}.

The techniques for the results in the previous paragraph are not well suited to the study of equilibrium states for $\varphi \neq 0$, which remains largely unexplored.  When 
the first version of the present work appeared [arXiv:1505.06371v1], 
 the only available references for this subject were existence results for a certain class of partially hyperbolic horseshoes \cite{LOR}, with uniqueness results only for potentials constant on the center-stable direction \cite{AL12}. An improved picture has emerged since then. 
Spatzier and Visscher studied uniqueness of equilibrium states for frame flows \cite{SV}.   Rios and Siqueira \cite{RS15} obtained uniqueness for H\"older potentials with small variation for certain partially hyperbolic horseshoes, and Ramos and Siqueira studied statistical properties of these equilibrium states \cite{RS16}.  Crisostomo and Tahzibi \cite{CT16} studied uniqueness of equilibrium states for partially hyperbolic DA systems on $\TT^3$, including the Ma\~n\'e family, using techniques very different from ours 
under the extra assumption that the potential is constant on `collapse intervals' of the semi-conjugacy.


The theory of SRB measures is much more developed. The fact that there is a unique SRB measure for the Ma\~n\'e diffeomorphisms follows from~\cite{BV}. The statistical properties of SRB measures is an active area of research \cite{AL, aC02, PSZ}.  
In \cite{CV}, interesting results are obtained on the continuity of the entropy of the SRB measure.

The characterization of the SRB measure as an equilibrium state for DA systems obtained along an arc of $C^\infty$  diffeomorphisms was established by Carvalho \cite{mC93}, with partial results in the $C^r$ case. However, the characterization of the SRB measure as a \emph{unique} equilibrium state is to the best of our knowledge novel for the Ma\~n\'e class.  Immediate consequences of this characterization include the upper large deviations principle and  multifractal analysis results of \S\ref{s.ldp}.


We now outline the paper. In \S\ref{s.back}, we give background material from \cite{CT4} on thermodynamic formalism.  In \S\ref{s.perturb}, we recall pressure estimates on $C^0$-perturbations of Anosov systems.  
In \S\ref{s.constructions}, we give details of the Ma\~n\'e construction. In \S\ref{s.mane.pf}, we prove Theorem \ref{t.mane}.  In \S \ref{sec:corollaries}, we 
prove Theorem \ref{cor1.2}. In \S\ref{s.srb}, we prove Theorem \ref{main3}.   In \S\ref{s.ldp}, we give results on large deviations and multifractal analysis.  In the preliminary sections \S\S~\ref{s.back}--\ref{s.perturb}, we follow the presentation of \cite{CFT_BV}, referring the reader to that paper for the proofs of some necessary background material.



\section{Background and preliminary results}\label{s.back}



\subsection{Pressure} 
Let $f\colon X\to X$ be a continuous map on a compact metric space.  We identify $X\times \mathbb{N}$ with the space of finite orbit segments by identifying $(x,n)$ with $(x,f(x),\dots,f^{n-1}(x))$.

Given a continuous potential function $\varphi\colon X\to \mathbb{R}$, write
$
S_n\varphi(x) = S_n^f \ph(x) = \sum_{k=0}^{n-1} \varphi(f^kx)
$.
For each $\eta>0$, write 
\[
\Var(\ph,\eta) = \sup\{|\ph(x)-\ph(y)| \,:\, x,y\in X, d(x,y) < \eta\}.
\]
Since we consider H\"older potentials, we will often use the bound
\[
\Var(\ph,\eta) \leq |\ph|_\alpha \eta^\alpha,
\text{ where }
|\ph|_\alpha := \sup_{x\neq y} \frac{|\ph(x)-\ph(y)|}{d(x,y)^\alpha}.
\]
The $n$th Bowen metric associated to $f$ is defined by
\[
d_n(x,y) = \max \{ d(f^kx,f^ky) \,:\, 0\leq k < n\}.
\]
Given $x\in X$, $\eps>0$, and $n\in \NN$, the \emph{Bowen ball of order $n$ with center $x$ and radius $\eps$} is
$
B_n(x,\eps) = \{y\in X \, :\,  d_n(x,y) < \eps\}.
$
A set $E\subset X$ is $(n,\eps)$-separated if $d_n(x,y) \geq \eps$ for all $x,y\in E$.

Given $\mathcal{D}\subset X\times \mathbb{N}$, we interpret $\DDD$ as a \emph{collection of orbit segments}. Write $\mathcal{D}_n = \{x\in X \, :\,  (x,n)\in \mathcal{D}\}$ for the set of initial points of orbits of length $n$ in $\mathcal{D}$.  Then we consider the partition sum
$$
\Lambda^{\mathrm{sep}}_n(\DDD,\ph,\epsilon; f) =\sup
\Big\{ \sum_{x\in E} e^{S_n\ph(x)} \, :\,  E\subset \mathcal{D}_n \text{ is $(n,\epsilon)$-separated} \Big\}.
$$
The \emph{pressure of $\ph$ on $\DDD$ at scale $\eps$} is 
$$
P(\mathcal{D},\ph,\epsilon; f) = \varlimsup_{n\to\infty} \frac 1n \log \Lambda^{\mathrm{sep}}_n(\mathcal{D},\ph,\epsilon),
$$
and the \emph{pressure of $\varphi$ on $\mathcal{D}$} is
$$
P(\mathcal{D},\ph; f) = \lim_{\epsilon\to 0}P(\mathcal{D},\ph,\epsilon).
$$

Given $Z \subset X$, let $P(Z, \varphi, \epsilon; f) := P(Z \times \NN, \varphi, \epsilon; f)$; observe that $P(Z, \varphi; f)$ denotes the usual upper capacity pressure \cite{Pesin}. We often write $P(\ph;f)$ in place of $P(X, \ph;f)$ for the pressure of the whole space.

When $\ph=0$, our definition gives the \emph{entropy of $\mathcal{D}$}:
\begin{equation}\label{eqn:h}
\begin{aligned}
h(\mathcal{D}, \epsilon; f)= h(\mathcal{D}, \epsilon) &:= P(\mathcal{D}, 0, \epsilon) \mbox{ and } h(\mathcal{D})= \lim_{\epsilon\rightarrow 0} h(\mathcal{D}, \epsilon).
\end{aligned}
\end{equation}

Write $\mathcal{M}(f)$ for the set of $f$-invariant Borel probability measures and $\mathcal{M}_e(f)$ for the set of ergodic measures in $\mathcal{M}(f)$.
The variational principle for pressure \cite[Theorem 9.10]{Wal} states that 
\[
P(\varphi;f)=\sup_{\mu\in \mathcal{M}(f)}\left\{ h_{\mu}(f) +\int \varphi \,d\mu\right\}
=\sup_{\mu\in \mathcal{M}_e(f)}\left\{ h_{\mu}(f) +\int \varphi \,d\mu\right\}.
\]
A measure achieving the supremum is an \emph{equilibrium state}.

\subsection{Obstructions to expansivity, specification, and regularity}

Bowen showed in \cite{Bow75} that if $(X,f)$ has expansivity and specification, and $\ph$ has a certain regularity property, then there is a unique equilibrium state. We recall definitions and results from \cite{CT4}, which show that non-uniform versions of Bowen's hypotheses suffice to prove uniqueness.

Given a homeomorphism $f\colon X\to X$, the \emph{bi-infinite Bowen ball around $x\in X$ of size $\eps>0$} is the set
\[
\Gamma_\eps(x) := \{y\in X \, :\,  d(f^kx,f^ky) < \eps \text{ for all } n\in \ZZ \}.
\]
If there exists $\eps>0$ for which $\Gamma_\eps(x)= \{x\}$ for all $x\in X$, we say $(X, f)$ is \emph{expansive}. When $f$ is not expansive, it is useful to consider the \emph{tail entropy} of $f$ at scale $\eps>0$ \cite{rB72,mM76}:
\begin{equation}\label{eqn:Bowen-tail-entropy}
h_f^*(\epsilon) = \sup_{x\in X} \lim_{\delta\to 0} \limsup_{n\to\infty}
\frac 1n \log \Lspan_n(\Gamma_\eps(x),\delta;f),
\end{equation}
where for a set $Y \subset X$, $\Lspan_n(Y, \delta;f) =\inf \{ \# E :\ Y \subset \bigcup_{x\in E} \overline{ B_n(x,\delta)} \}$.

\begin{defn} \label{almostexpansive}
For $f\colon X\rightarrow X$ the set of non-expansive points at scale $\epsilon$ is  $\mathrm{NE}(\epsilon):=\{ x\in X \, :\,  \Gamma_\eps(x)\neq \{x\}\}$.  An $f$-invariant measure $\mu$ is  almost expansive at scale $\epsilon$ if $\mu(\mathrm{NE}(\epsilon))=0$.  Given a potential $\varphi$, the pressure of obstructions to expansivity at scale $\epsilon$ is
\begin{align*}
\Pexp(\varphi, \epsilon) &=\sup_{\mu\in \mathcal{M}_e(f)}\left\{ h_{\mu}(f) + \int \varphi\, d\mu\, :\, \mu(\mathrm{NE}(\epsilon))>0\right\} \\
&=\sup_{\mu\in \mathcal{M}_e(f)}\left\{ h_{\mu}(f) + \int \varphi\, d\mu\, :\, \mu(\mathrm{NE}(\epsilon))=1\right\}.
\end{align*}
This is monotonic in $\eps$, so we can define a scale-free quantity by
\[
\Pexp(\varphi) = \lim_{\epsilon \to 0} \Pexp(\varphi, \epsilon).
\]
\end{defn}

\begin{defn} 
A collection of orbit segments $\mathcal{G}\subset X\times \mathbb{N}$ has \emph{specification at scale $\epsilon$} if there exists $\tau\in\mathbb{N}$ such that for every $\{(x_j, n_j)\, :\, 1\leq j\leq k\}\subset \mathcal{G}$, there is a point $x$ in
$$\bigcap_{j=1}^k f^{-(m_{j-1}+ \tau)}B_{n_j}(x_j, \epsilon),$$
where $m_{0}=-\tau$ and $m_j = \left(\sum_{i=1}^{j} n_i\right) +(j-1)\tau$ for each $j \geq 1$.
\end{defn}

The above definition says that there is some point $x$ whose trajectory shadows each of the $(x_i,n_i)$ in turn, taking a transition time of exactly $\tau$ iterates between each one.  The numbers $m_j$ for $j\geq 1$ are the time taken for $x$ to shadow $(x_1, n_1)$ up to $(x_j, n_j)$.

\begin{defn} \label{Bowen}
Given $\mathcal{G}\subset X\times \mathbb{N}$, a potential $\varphi$  has the \emph{Bowen property on $\GGG$ at scale $\epsilon$} if
\[
V(\GGG,\ph,\epsilon) := \sup \{ |S_n\varphi (x) - S_n\varphi(y)| : (x,n) \in \GGG, y \in B_n(x, \epsilon) \} <\infty.
\]
We say $\varphi$ has the \emph{Bowen property on $\GGG$} if there exists $\epsilon>0$ so that $\varphi$ has the Bowen property on $\GGG$ at scale $\epsilon$.
\end{defn}

We refer to an upper bound for $V(\GGG,\ph,\eps)$ as a \emph{distortion constant}.
Note that if $\GGG$ has the Bowen property at scale $\eps$, then it has it for all smaller scales. 


\subsection{General results on uniqueness of equilibrium states}

Our main tool for existence and uniqueness of equilibrium states is \cite[Theorem 5.5]{CT4}.

\begin{defn}
A \emph{decomposition} for $(X,f)$ consists of three collections $\mathcal{P}, \mathcal{G}, \mathcal{S}\subset X\times (\NN\cup\{0\})$ and three functions $p,g,s\colon X\times \mathbb{N}\to \NN\cup\{0\}$ such that for every $(x,n)\in X\times \NN$, the values $p=p(x,n)$, $g=g(x,n)$, and $s=s(x,n)$ satisfy $n = p+g+s$, and 
\begin{equation}\label{eqn:decomposition}
(x,p)\in \mathcal{P}, \quad (f^p(x), g)\in\mathcal{G}, \quad (f^{p+g}(x), s)\in \mathcal{S}.
\end{equation}
Given a decomposition $(\PPP,\GGG,\SSS)$ and $M\in \mathbb{N}$, we write $\GGG^M$ for the set of orbit segments $(x,n)$ for which $p \leq M$ and $s\leq M$.
\end{defn}

Note that the symbol $(x,0)$ denotes the empty set, and the functions $p, g, s$ are permitted to take the value zero. 
 
\begin{thm}[Theorem 5.5 of \cite{CT4}]\label{t.generalM}
Let $X$ be a compact metric space and $f\colon X\to X$ a homeomorphism. 
Let $\ph \colon X\to\RR$ be a continuous potential function.
Suppose that $\Pexp(\ph) < P(\ph)$, and that $(X,f)$ admits a decomposition $(\PPP, \GGG, \SSS)$ with the following properties:
\begin{enumerate}
\item $\GGG$ has specification at any scale;
\item  $\ph$ has the Bowen property on $\GGG$;
\item $P(\PPP \cup \SSS,\ph) < P(\ph)$.
\end{enumerate}
Then there is a unique equilibrium state for $\ph$.
\end{thm}


\section{Perturbations of Anosov Diffeomorphisms}\label{s.perturb}
We collect some background material about weak forms of hyperbolicity and perturbations of Anosov diffeomorphisms.

\subsection{Partial hyperbolicity} \label{basics}

Let $M$ be a compact manifold.  Recall that a diffeomorphism $f\colon M\to M$ is \emph{partially hyperbolic} if there is a $Df$-invariant splitting $TM=E^s\oplus E^c\oplus E^u$ and constants $N\in \NN$, $\lambda>1$ such that for every $x\in M$ and every unit vector $v^\sigma\in E^\sigma$ for $\sigma\in \{s, c, u\}$, we have
\begin{enumerate}[label=(\roman{*})]
\item $\lambda \|Df^N_x v^s\|<\|Df^N_x v^c\|<\lambda^{-1}\|Df^N_x v^u\|$, and
\item $\|Df^N_x v^s\|<\lambda^{-1}<\lambda<\|Df^N_x v^u\|$.
\end{enumerate}

A partially hyperbolic diffeomorphism $f$ admits \emph{stable and unstable foliations} $W^s$ and $W^u$, which are $f$-invariant and tangent to $E^s$ and $E^u$, respectively \cite[Theorem 4.8]{yP04}.  There may or may not be foliations tangent to either $E^c$, $E^s\oplus E^c$, or $E^c\oplus E^u$.  When these exist we denote these by $W^c$, $W^{cs}$, and $W^{cu}$ and refer to these as the center, center-stable, and center-unstable foliations respectively.  For $x\in M$, we let $W^{\sigma}(x)$ be the leaf of the foliation $\sigma\in \{s, u, c, cs, cu\}$ containing $x$ when this is defined.

For a foliation $W$, we write  $d_W$ for the leaf metric, and write $W_\eta(x)$ for the $d_W$-ball of radius $\eta$ in $W(x)$. 
Suppose $W^1,W^2$ are foliations of $M$ with the property that $TM = TW^1 \oplus TW^2$.   
We say that $W^1,W^2$ have a \emph{local product structure at scale $\eta>0$ with constant $\kappa\geq 1$} if for every $x,y\in M$ with $\eps := d(x,y) < \eta$, the leaves $W^1_{\kappa \eps}(x)$ and $W^2_{\kappa \eps}(y)$ intersect in a single point.

\subsection{Anosov shadowing lemma} \label{constants}

The Anosov shadowing lemma is proved in e.g. \cite[Theorem 1.2.3]{sP99}.

\begin{lem}[Anosov Shadowing Lemma]\label{shadowinglemma}
Let $f$ be a transitive Anosov diffeomorphism. There is $C=C(f)$ so that if $2 \eta>0$ is an expansivity constant for $f$, then every $\frac\eta C$-pseudo-orbit 
for $f$ can be $\eta$-shadowed by an orbit 
for $f$. 
\end{lem}
The following result is proved in \cite[Lemma 3.2]{CFT_BV} using the natural semi-conjugacy which exists for maps in a $C^0$ neighborhood of $f$ as a consequence of the Anosov shadowing lemma.

\begin{lem} \label{pressuredrop}
Let $f$ be a transitive Anosov diffeomorphism, $C= C(f)$ the constant from the shadowing lemma, and $3 \eta>0$ an expansivity constant for $f$.  If $g\in \Diff(M)$ is such that $d_{C^0}(f,g) < \eta/C$, then:
\begin{enumerate}[label=\textup{(\roman{*})}]
\item\label{Pg-geq}
 $P(\varphi; g)\geq P(\varphi;f)- \Var(\varphi, \eta)$;
\item\label{Lambdag-leq}
 $\Lambda^{\mathrm{sep}}_n(\varphi, 3 \eta ;g) \leq \Lambda^{\mathrm{sep}}_n(\varphi, \eta ;f)e^{n \Var(\varphi, \eta) }$.
\end{enumerate}
\end{lem}

In particular, \ref{Lambdag-leq} gives
$P(\varphi, 3\eta; g) \leq P(\varphi; f) + \Var(\varphi, \eta)$.

\subsection{Pressure estimates} \label{pressureestimate}

The Ma\~n\'e examples are $C^0$ perturbations of Anosov maps, where the perturbation is made inside a neighborhood of a fixed point $q$.  
We estimate the pressure of orbit segments spending nearly all their time near $q$. 

Let $f$ be a transitive Anosov diffeomorphism of a compact manifold $M$, with topological entropy $h=\htop(f)$ and expansivity constant $3\eta$.  Let $C$ be the constant from the shadowing lemma. For any $\eta>0$  smaller than the expansivity constant for $f$,  let $L =  L(f, \eta)$ be a constant so that for every $n$,
\begin{equation}\label{eqn:Lambda-upper-bound}
\Lambda^{\mathrm{sep}}_n(0,\eta; f) \leq  L e^{nh}.
\end{equation} 
This is possible by \cite[Lemma 3]{Bow75}. Let $g\colon M\to M$ be a diffeomorphism with $d_{C^0}(f,g) < \eta/C$.  Given a fixed point $q$ of $f$ and a scale $\rho \in (0,3\eta)$, let 
$\chi_q$ be the indicator function of $M\setminus B(q,\rho)$, and 
 consider the following collection of orbit segments for $g$:
\[
\CCC = \CCC(g,q,r) = \{(x,n) \in M\times \NN : S_n^g \chi_q(x) < rn \},
\]
The following estimates are proved in \cite[Theorem 3.3]{CFT_BV}.

\begin{thm} \label{coreestimateallscales}
Under the assumptions above, we have the inequality 
\begin{equation} \label{eqn:hC}
h(\CCC,6\eta;g) \leq r(h + \log L) +H(2r),
\end{equation}
where $H(t) = -t\log t - (1-t)\log(1-t)$.  Moreover, given $\ph\colon M\to\RR$ continuous and $\delta>0$, we have
\begin{equation} \label{eqn:PC}
P(\CCC,\ph,\delta;g) \leq (1-r) \sup_{B(q, \rho)}\ph + r \sup_{M}\ph + h(\CCC,\delta;g),
\end{equation}
and thus it follows that
\[
P(\CCC, \ph; g) \leq  h_g^\ast(6 \eta) +(1-r) \sup_{B(q, \rho)} \ph + r( \sup_{M}\ph + h + \log L ) + H(2r).
\]

\end{thm}

\subsection{Obstructions to expansivity}\label{sec:obstr-exp}

Let $g$ be as in the previous section, and suppose that the following property \ref{E} holds.  
\begin{enumerate}[label=\textbf{[\Alph{*}{]}}]
\setcounter{enumi}{4}
\item\label{E} there exist $\epsilon >0$, $r>0$, and a fixed point $q$ such that for $x \in M$, if there exists a sequence $n_k\to\infty$ with $\frac{1}{n_k}S^g_{n_k}\chi_q(x) \geq r$, then $\Gamma_\eps(x)=\{x\}$.
\end{enumerate}

Then $\CCC = \CCC(r)$ from above has the following property, which is proved in \cite[Theorem 3.4]{CFT_BV}.

\begin{thm} \label{expansivityestimate}
Under the above assumptions, we have the pressure estimate $\Pexp(\ph,\eps) \leq P(\CCC(q,r),\ph)$.
\end{thm}

Let $\chi=\chi_q$ and $\CCC=\CCC(q,r;g)$.  Consider the set
$$
A^+ = \{x : \text{there exists } K(x) \text{ so } \tfrac{1}{n}S^g_{n}\chi(x) < r \text{ for all } n>K(x)\}
$$
The next lemma, which we need in \S \ref{s.srb}, is proved as \cite[Lemma 3.5]{CFT_BV}  as an intermediate step in the proof of Theorem \ref{expansivityestimate}.  

\begin{lem} \label{keystepexpansivityestimate}
Let $\mu \in \mathcal{M}_e(g)$. If $\mu(A^+)>0$, then 
$h_{\mu}(g)+\int \ph\, d \mu \leq P(\CCC, \ph)$.
\end{lem}

\subsection{Cone estimates and local product structure}\label{sec:torus}

Let $F^1,F^2 \subset \RR^d$ be subspaces such that $F^1 \cap F^2 = \{0\}.$ Let
$\measuredangle(F^1,F^2) := \min\{\measuredangle(v,w) \, :\,  v\in F^1, w\in F^2\}$, and define
\begin{equation}\label{eqn:barkappa}
\bk(F^1,F^2) :=  (\sin\measuredangle(F^1,F^2))^{-1} \geq 1.
\end{equation}
Given $\beta \in (0,1)$ and $F^1,F^2 \subset \RR^d$, the \emph{$\beta$-cone of $F^1$ and $F^2$} is 
$$
C_\beta(F^1,F^2) = \{ v+w \, :\,  v\in F^1, w\in F^2, \|w\| < \beta \|v\| \}.
$$
The following two useful lemmas are proved in \S 8 of \cite{CFT_BV}.
\begin{lem}\label{lem:Wlps}
Let $W^1,W^2$ be any foliations of $F^1 \oplus F^2$ with $C^1$ leaves such that $T_x W^1(x) \subset C_\beta(F^1,F^2)$ and $T_x W^2(x) \subset C_\beta(F^2,F^1)$, and let $\bk = \bk(F^1,F^2)$.  Then for every $x,y\in F^1 \oplus F^2$ the intersection $W^1(x) \cap W^2(y)$ consists of a single point $z$.  Moreover, 
\[
\max\{d_{W^1}(x,z), d_{W^2}(y,z)\} \leq \frac{1+\beta}{1-\beta} \bk d(x,y).
\]
\end{lem}

\begin{lem}\label{compare:dist}
Under the assumptions of Lemma \ref{lem:Wlps}, suppose that $x, y$ are points belonging to the same local leaf of $W\in \{W^1, W^2\}$. Then
\[
d(x,y) \leq d_W(x,y) \leq (1+\beta)^2 d(x,y).
\]
\end{lem}

\section{Construction of Ma\~n\'e's examples}\label{s.constructions}

We review the class of robustly transitive diffeomorphisms originally considered by Ma\~{n}\'{e}~\cite{man78}.
Fix $d\geq3$ and let $f_A$ be the hyperbolic automorphism of $\mathbb{T}^d$ determined by a matrix $A\in\mathrm{SL}(d,\mathbb{Z})$ with  all eigenvalues real, positive, simple, and irrational and only one eigenvalue outside the unit circle.
Let $\lambda_u$ be the unique eigenvalue greater than $1$ and $\lambda_s<1$ be the largest
of the other eigenvalues. Let $h=\htop(f_A)$ be the topological entropy. 

The Ma\~{n}\'{e} class of examples are $C^0$ perturbations of $f_A$, which we will denote by $f_M$.  We describe the construction below. We are careful about issues of scale to guarantee that we have local product structure at a scale which is `compatible' with the $C^0$ size of the perturbation.  

Fix an expansivity constant $3\eta$ for $f_M$.  We require that $\eta$ is small enough so that calculations at scales which are a suitable multiple of $\eta$ are local: a necessary upper bound on $\eta$ can be computed explicitly, depending on basic properties of the map $f_M$. 
 Let $q$ be a fixed point for $f_A$, and fix $0< \rho < 3\eta$.  We carry out a perturbation in a $\rho$-neighborhood of $q$.

Let $F^u,F^c,F^s \subset \RR^d$ be the eigenspaces corresponding to (respectively) $\lambda_u$, $\lambda_s$, and all eigenvalues smaller than $\lambda_s$, and let $F^{cs} = F^c\oplus F^s$.   
Let $\kappa = 2\bk(F^s,F^u)$, where $\bk$ is as in \eqref{eqn:barkappa}.

Let $\mathcal{F}^{u,c,s}$ be the foliations of $\TT^d$ by leaves parallel to $F^{u,c,s}$.  
These leaves 
are dense in $\mathbb{T}^d$ since all eigenvalues are irrational. 
Let 
$\beta\in (0, \rho)$ be sufficiently small and
consider the cones
\begin{alignat*}{3}
C_\beta^s &= C_\beta(F^s, F^{cu}), &\qquad
C_\beta^c &= C_\beta(F^c, F^s \oplus F^u), \\
C_\beta^u &= C_\beta(F^u, F^{cs}), &
C_\beta^{cs} &= C_\beta(F^{cs},F^u).
\end{alignat*}

\begin{figure}[htb]
\begin{center}
\ifps
\psfrag{A}{$f_A$}
\psfrag{B}{$f_M$}
\psfrag{q}{$q$}
\psfrag{r}{$q_1$}
\psfrag{s}{$q$}
\psfrag{t}{$q_2$}
\includegraphics[width=.6\textwidth]{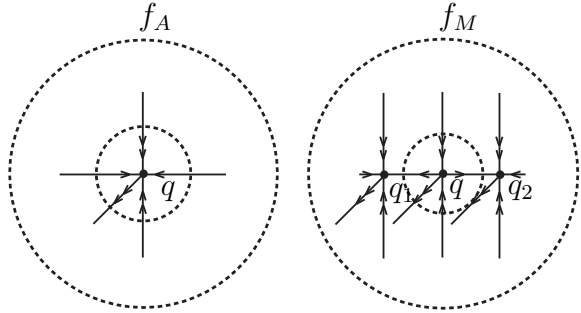}
\fi
\caption{Ma\~{n}\'{e}'s construction}\label{f.mane}
\end{center}
\end{figure}

Outside of $B(q,\rho)$, set $f_M$ to be equal to $f_A$. Inside $B(q,\rho)$, the fixed point $q$ undergoes a pitchfork bifurcation in the direction of $\mathcal{F}^c$; see \cite{man78} for details.  
The perturbation is carried out so that 
\begin{itemize}
\item $\mathcal{F}^c$ is still an invariant foliation for $f_M$, and we write $E^c = T\mathcal{F}^c$;
\item 
the cones $C_\beta^u$ and $C_{\beta}^s$ are invariant and uniformly expanding under $Df_M$ and $Df_M^{-1}$, respectively; in particular, they contain $Df_M$-invariant distributions $E^s$ and $E^u$ that integrate to $f_M$-invariant foliations $W^s$ and $W^u$.
\item $E^{cs}:= E^c \oplus E^s$ integrates to a foliation $W^{cs}$. This holds because $E^s \subset C_{\beta}^s$ guarantees that $E^{cs} \subset C_\beta^{cs}$.
\end{itemize} 
Thus, $f_M$ is partially hyperbolic with $T\TT^d = E^s \oplus E^c \oplus E^u$ and $E^{cs}$ integrates to a foliation.  

The index of $q$ changes during the perturbation, and we may also assume that for any point in $\mathbb{T}^d\setminus B(q,\rho/2)$ the contraction in the direction $E^c$ is $\lambda_s$. 
Inside $B(q,\rho/2)$, the perturbed map experiences some weak expansion in the direction $E^c$, and two new fixed points are created on $W^c(q)$, see Figure~\ref{f.mane}. Let $\lambda= \lambda_{c}(f_M) >1$ be the greatest expansion which occurs in the center direction. We can carry out the construction so that $\lambda$ is arbitrarily close to $1$. 

Since $f$ contracts $E^{cs}$ by a factor of at least $\lambda_s$ outside $B(q,\rho/2)$, and expands it by at most $\lambda$ inside the ball, we can estimate $\|Df_M^n|_{E^{cs}(x)}\|$ by counting how many of the iterates $x, f(x), \dots, f^{n-1}(x)$ lie outside $B(q,\rho/2)$.  If at least $\second n$ of these iterates lie outside the ball, then
\begin{equation}\label{eqn:gamma-1}
\|Df_M^n|_{E^{cs}(x)}\| \leq \lambda_s^{\second n} \lambda^{(1-\second)n}.
\end{equation}
Thus we are interested in a value of $\second>0$ that gives $\lambda_s^\second \lambda^{1-\second} < 1.$ Consider the quantity
\[
\gamma = \gamma(f_M) =\frac{\ln \lambda }{\ln \lambda - \ln \lambda_s} >0.
\]
Then $\gamma \to 0$ as $\lambda \to 1$, and for $r>\gamma$ a simple calculation gives
\begin{equation}\label{theta}
\theta_r:= \lambda_s^\second \lambda^{1-\second} < 1.
\end{equation}

Given $\rho,\second>0$, we write $\FFF_{\rho,\second}$ for the set of Ma\~n\'e diffeomorphisms constructed as described here for which $\gamma(f_M)<r$. Thus, for $f \in \FFF_{\rho,\second}$, we have $\theta_r(f_M)<1$.

There is a constant $K$ so that we can carry out the construction to satisfy $d_{C^0}(f_M,f_A) < K\rho$, $f_A(B(q,\rho)) \subset B(q, K \rho)$, and $f_M(B(q,\rho)) \subset B(q, K \rho)$.
In particular, by choosing $\rho$ small, we can ensure that  $d_{C^0}(f_M, f_A) < \eta/C$ where $C=C(f_A)$ is the constant from the Shadowing Lemma. 

We now consider diffeomorphisms $g$ in a $C^1$ neighborhood of $f_M$.  
For sufficiently small $C^1$ perturbations $g$ of $f_M$, the following remain true.
\begin{itemize}
\item $d_{C^0}(g, f_A) < \eta/C$, where $C=C(f_A)$ is the constant from Lemma \ref{shadowinglemma}.
\item $g$ is partially hyperbolic with $T\TT^d = E^s_g \oplus E^c_g \oplus E^u_g$, where $E^\sigma_g \subset C_\beta^\sigma$ for each $\sigma\in \{s,c,u,cs\}$.
\item 
The distributions $E^c_g$ and $E^{cs}_g$ integrate to foliations $W^c_g$ and $W^{cs}_g$.
\item Each of the leaves $W^{cs}_g(x)$ and $W^u_g(x)$ is dense for every $x\in \TT^d$.
\end{itemize}

For the $C^1$ perturbations, partial hyperbolicity with $E_g^\sigma\subset C_\beta^\sigma$ and integrability are provided by \cite[Theorem 6.1]{HPS};
density of the leaves was shown in \cite{PS}.  Given $g$ as above, let
\begin{align*}
\lambda_c(g) &= \sup \{ \|Dg|_{E^c(x)}\| : x\in B(q,\rho/2)\}, \\
\lambda_s(g) &= \sup \{ \|Dg|_{E^c(x)}\| : x \in \TT^d \setminus B(q,\rho/2)\}, \\
\gamma (g)  &=\frac{\ln \lambda_c(g) }{\ln \lambda_c(g) - \ln \lambda_s(g)}.
\end{align*}
Let $\UUU_{\rho,\second}$ be the set of $C^1$ diffeomorphisms $g\colon \TT^d\to \TT^d$ satisfying the conditions in the list above with $\gamma(g)<r$. A simple calculation gives
\begin{equation}\label{eqn:gamma-r}
\theta_r(g) := \lambda_c(g)^{1-r}\lambda_s(g)^r < 1.
\end{equation}

\section{Proof of Theorem~\ref{t.mane}}\label{s.mane.pf}

We let $g \in \UUU_{\rho, \second}$, and consider the collection $\GGG$ of orbit segments $(x,n)$ for which $(x,i)$ spends at least $\gamma i$ iterates outside of $B(q,\rho)$ for all $i\leq n$. 
We will show that these orbit segments experience uniform contraction  in the $E^{cs}$ direction.  Using local product structure, 
this will allow us to prove specification and the Bowen property for such orbit segments. 
The hypothesis $\Psi(\rho,r,\varphi)< P(\varphi; g)$, together with Theorems \ref{coreestimateallscales} and \ref{expansivityestimate}, allow us to bound the pressure of obstructions to expansivity and specification away from $P(\ph; g)$.

\subsection{Local product structure}\label{lps}

We require local product structure for $g$ at scale $6 \eta$ repeatedly through this section.  This holds because the splitting for $g$ is contained in thin cone fields and so the local leaves are near the local leaves for $f_A$ when $\beta$ and $\eta$ are small.  
\begin{lem}\label{lem:lps-mane}
The diffeomorphism $g$ has local product structure for $W^{cs}_g,W^{u}_g$ at scale $6\eta$ with constant $\kappa = 2\bk(F^{s},F^u)$.
\end{lem}

\begin{proof}
Let $\widetilde W^{cs}$ and $\widetilde W^{u}$ be the lifts of $W^{cs},W^{u}$ to $\RR^d$.  Given $x,y\in \TT^d$ with $\eps := d(x,y) < 6\eta$, let $\tilde x,\tilde y\in \RR^d$ be lifts of $x,y$ with $\eps=d(\tilde x,\tilde y)<6\eta$.  By Lemma \ref{lem:Wlps} the intersection $\widetilde W^{cs}(x) \cap \widetilde W^{u}(y)$ has a unique point $\tilde z$, which projects to $z\in \TT^d$.  Moreover, the leaf distances between $\tilde x,\tilde z$ and $\tilde y, \tilde z$ are at most $(\frac{1+\beta}{1-\beta}) \bk(F^{s},F^u)\eps$.  Since $\beta$ is small, this is less than $2\bk(F^{s},F^u) \epsilon$, so $z \in \tilde W^{cs}_{\kappa \eps}(x) \cap \tilde W^u_{\kappa \eps}(y)$.

{ By choosing $\eta$ not too large, we can ensure that $6 \eta \kappa$ is not too large relative to the diameter of $\TT^d$, so that the projection of $\widetilde W^{cs}_{6 \eta \kappa}(x) \cap \widetilde W^{u}_{6 \eta \kappa}(y)$ coincides with $W^{cs}_{6 \eta \kappa}(x) \cap W^{u}_{6 \eta \kappa}(y)$. Thus,  $z$ is the only point in this intersection.}
\end{proof}

\subsection{Specification}

A main ingredient for establishing specification for mixing locally maximal hyperbolic sets $f\colon \Lambda\to \Lambda$ is that given $\delta>0$, there exists $N\in\mathbb{N}$ such that for $x,y\in \Lambda$ and $n\geq N$ we have $f^n(W^u_{\delta}(x))\cap W^s_{\delta}(y)\neq \emptyset$.  We mimic this idea replacing the stable manifold with the centerstable manifolds. All leaves of $W^u$ are dense in $\TT^d$ by the definition of $\UUU_{\rho, \second}$. The following lemma gives uniform density.

\begin{lem}\label{lem:intersection}
For every $\delta>0$ there is $R>0$ such that for every $x,y\in \TT^d$, we have $W_R^u(x) \cap W_\delta^{cs}(y) \neq \emptyset$.
\end{lem}

\begin{proof}
 Fix $\delta>0$ and let $\alpha=\delta/\kappa$, where $\kappa$ is the constant from the local product structure. Fix $R_0>0$ such that each unstable leaf is $\alpha$-dense in the manifold. Thus for every $x\in \TT^d$ there is $z\in W_{R_0}^u(x)$ such that $d(y,z)<\alpha$, so by local product structure, $W^u_\delta(z)\cap W_\delta^{cs}(y)\neq \emptyset.$ Thus,  $W_{R_0+\delta}^u(x) \supset W_\delta^u(z)$ and so writing $R=R_0 +\delta$, we have $W^u_R(x)\cap W^{cs}_\delta(y)\neq\emptyset$.
\end{proof}

Because $g$ is uniformly expanding along $W^u$, we see that for every $\delta>0$ there is $N\in \NN$ such that for every $x\in \TT^d$ and $n\geq N$, we have $g^n(W^u_\delta(x)) \supset W^u_R(g^nx)$.  Thus by Lemma \ref{lem:intersection} we have
\begin{equation}\label{eqn:iterated-intersection}
g^n(W^u_\delta(x))\cap W^{cs}_\delta(y)\neq \emptyset \text{ for every } x,y\in \TT^d.
\end{equation}

Let $\chi$  be the indicator function of $\TT^d \setminus B(q, \rho)$, so that
$\frac{1}{i}S_i \chi(x)$
is the proportion of time that an orbit segment $(x, i)$ spends outside $B(q, \rho)$.

\begin{lem} \label{centrestable} 
Suppose $(x,n)\in \TT^d\times \NN$ is such that $\frac{1}{i}S_i\chi(x)\geq r$ for all $0\leq i\leq n$, and $\theta_r\in (0,1)$ is the constant defined at \eqref{eqn:gamma-r}. Then
\begin{enumerate}[label=(\alph{*})]
\item For any $y\in B_n(x,\rho/2)$, we have $\|Dg^i|_{E^{cs}(y)}\| < (\theta_r)^i$ for all  $0\leq i\leq n$.
\item For any $y,z\in W_{\rho/2}^{cs}(x)$, we have $d_W(f^iy,f^iz) \leq \theta_r^i d_W(y,z)$ for all $0\leq i\leq n$. 
\item For $0<\delta<\rho/2$, we have $W^{cs}_{\delta}(x)\subset B_n(x,2 \delta)$.
\end{enumerate}
\end{lem}
\begin{proof}
Given $0\leq i\leq n$, the inequality $\frac{1}{i}S_i\chi(x)>r$ implies that the orbit segment $(x,i)$ spends at least $ir$ iterates outside of $B(q,\rho)$.  It follows that $(y,i)$ spends at least $ir$ iterates outside of $B(q,\rho/2)$.  By the definition of $\lambda_c(g)$ and $\lambda_s(g)$, it follows that
\[
\|Dg^i|_{E^{cs}(y)}\| \leq \lambda_c^{i - ir} \lambda_s^{ir} = (\theta_r)^i,
\]
proving the first claim. It is an easy exercise to prove (b) using the uniform contraction estimate provided by (a), and (c) follows immediately from (b) and Lemma \ref{compare:dist} (using that $\beta$ is small so $(1+\beta)^2<2$).
\end{proof}

Now we define the decomposition.  We consider the following collections of orbit segments:
\begin{equation}\label{eqn:mane-decomp}
\begin{aligned}
\mathcal{G}&=\{(x,n)\in \mathbb{T}^d\times \mathbb{N}\, :\, S_i\chi(x)\geq ir\, \, \forall\,\, 0\leq i\leq n\},\\
\PPP &= \{(x,n)\in \mathbb{T}^d\times \mathbb{N}\, :\, S_n\chi(x)< nr \}.
\end{aligned}
\end{equation}

The collection $\GGG$ is chosen so that the centerstable manifolds are uniformly contracted along orbit segments from $\GGG$. These collections, together with the trivial collection $\{(x,0) : x \in X\}$ for $\SSS$, define a decomposition of any point $(x,n)\in X\times \mathbb{N}$ as follows:  let $p$ be the largest integer in $\{0,..., n\}$ such that $\frac{1}{p}S_p\chi(x)<r$, and thus $(x, p) \in \PPP$. We must have $(g^p(x), n-p)\in\mathcal{G}$ since if $\frac{1}{k}S_k\chi(g^px)<r$ for some $0\leq k\leq n-p$, then 
$\frac{1}{p+k}S_{p+k}\chi(x) = \frac 1{p+k} \left(S_p\chi(x) + S_k\chi(g^p(x))\right) < r,$ contradicting the maximality of $p$.

\begin{lem} \label{spec}
The collection $\GGG$ has specification at any scale $\delta>0$.
\end{lem}
\begin{proof}
For an arbitrary fixed $\delta>0$, we prove specification at scale $3\delta$. The key property that allows us to transition from one orbit to another is \eqref{eqn:iterated-intersection}.  This property, together with uniform expansion on $W^u$, allows us to choose $\tau = \tau(\delta)\in \NN$ such that 
\[
\begin{aligned}
g^\tau(W^u_\delta(x))&\cap W^{cs}_\delta(y)\neq \emptyset \text{ for all } x,y\in \TT^d, \\
d(g^{-\tau}y, g^{-\tau} z) &< \frac{1}{2}d(y, z) \text{ for all } x\in \TT^d \text{ and } y,z\in W^u_\delta(x).
\end{aligned}
\]
Given any $(x_1, n_1), \dots, (x_k, n_k)\in \mathcal{G}$, we construct $y_j$ such that $(y_j,m_j)$ shadows $(x_1,n_1),\dots,(x_j,n_j)$, where $m_1 =n_1$, $m_2 = n_1 + \tau + n_2$, $\dots$, $m_k = (\sum_{i=1}^{k} n_i) + k\tau$. We also set $m_{0} = - \tau$.
 
Let $y_1 = x_1$, and choose $y_2,\dots, y_k$ recursively so that
$$
\begin{matrix}
g^{m_1}y_1 \in W^{u}_\delta (g^{m_1}y_1) &  \mbox{and} & g^{m_1+\tau}y_2 \in W^{cs}_\delta (x_2)\\
g^{m_2}y_3 \in W^{u}_\delta (g^{m_1}y_2) &  \mbox{and}   & g^{m_1+\tau}y_3 \in W^{cs}_\delta (x_3)\\
\vdots & \vdots & \vdots \\
g^{m_{k-1}}y_k \in W^{u}_\delta (g^{m_{k-1}}y_{k-1}) &  \mbox{and}   & g^{m_{k-1}+\tau}y_{k} \in W^{cs}_\delta (x_k).\\
\end{matrix}
$$
Since $g^{m_j}y_{j+1}$ is in the unstable manifold of $g^{m_j}y_j$,  and distance is contracted by $\frac{1}{2}$ every time the orbit passes backwards through a `transition', we obtain that 
$$
\begin{matrix}
d_{n_j}(g^{m_{j-1}+\tau}y_j, g^{m_{j-1}+\tau}y_{j+1})& < &\delta \\
d_{n_{j-1}}(g^{m_{j-2}+\tau}y_j, g^{m_{j-2}+\tau}y_{j+1})& < &\delta/2 \\
\vdots &  &\vdots\\
d_{n_1}(y_j, y_{j+1}) & < & \delta/2^j.
\end{matrix}
$$
That is, $d_{n_{j-i}}(g^{m_{j-i-1}+\tau}y_j, g^{m_{j-i-1}+\tau}y_{j+1}) < \delta/2^i$ for each  $i \in \{1, \ldots, j\}$. Since $g^{m_j+\tau}(y_{j+1}) \in B_{n_{j+1}}(x_{j+1}, \delta)$ by Lemma \ref{centrestable}, it follows that $$d_{n_j}(g^{m_{j-1}+ \tau}y_k, x_j) < 2 \delta + \sum_{j=1}^\infty 2^{-j} \delta = 3\delta.$$
Thus, $y_k\in \bigcap_{j=1}^k g^{-(m_{j-1} + \tau)}B_{n_j}(x_j, 3 \delta)$, and so $\mathcal{G}$ has specification at scale $3 \delta$.
\end{proof}

\subsection{The Bowen property}

Let $\theta_u \in (0,1)$ be such that $\|Dg|_{E^u(x)}^{-1}\| \leq \theta_u$ for all $x\in \TT^d$. 
Let $\kappa$ be the constant associated with the local product structure of $E^{cs}_g \oplus E^u_g$. Let $\eps = \rho/(2\kappa )$.
 
\begin{lem}\label{bowen-balls}
Given $(x,n)\in \GGG$ and $y\in B_n(x,\eps)$, we have
\begin{equation}\label{eqn:hyp-hyp}
d(g^kx,g^ky) \leq \kappa \eps(\theta_r^k + \theta_u^{n-k})
\end{equation}
for every $0\leq k\leq n$.
\end{lem}
\begin{proof}
Using the local product structure,  there exists $z\in W^{cs}_{\kappa \eps}(x) \cap W^u_{\kappa \eps}(y)$.  Since $g^{-1}$ is uniformly contracting on $W^u$, we get 
\[
d(g^kz,g^ky) \leq \theta_u^{n-k} d(g^nz,g^ny) \leq \theta_u^{n-k} \kappa \eps,
\]
and Lemma \ref{centrestable} gives $d(g^kx,g^kz) \leq \theta_r^k d(x,z) \leq \theta_r^k \kappa \eps$. The triangle inequality gives \eqref{eqn:hyp-hyp}.
\end{proof}

\begin{lem} \label{bowenprop}
Any H\"older continuous $\varphi$ has the Bowen property on $\GGG$ at scale $\epsilon$.
\end{lem}
\begin{proof}
Since $\varphi$ is H\"older,  there exists $K>0$ and $\alpha\in (0,1)$ such that $|\ph(x) - \ph(y)| \leq K d(x,y)^\alpha$ for all $x,y\in \TT^d$.  For $(x, n) \in \GGG$ and $y\in B_n(x,\eps)$, Lemma \ref{bowen-balls} gives
\[
|S_n\ph(x)-S_n \ph(y)| \leq K\sum_{k=0}^{n-1}d(g^kx, g^ky)^\alpha \leq K(\kappa \eps)^\alpha \sum_{k=0}^{n-1} (\theta_u^{n-k} + \theta_r^{k})^\alpha.
\]
The summand admits the upper bound
\[
(\theta_u^{n-k} + \theta_r^k)^\alpha \leq (2\theta_u^{n-k})^{\alpha} + (2\theta_r^k)^{\alpha},
\]
and we conclude that
\[
|S_n\ph(x)-S_n \ph(y)| \leq  K(2\kappa\eps)^\alpha \sum_{j=0}^\infty (\theta_u^{j\alpha} + \theta_r^{j\alpha})  < \infty.
\]
\end{proof}

\subsection{Expansivity}

The diffeomorphism $g$ is partially hyperbolic with one-dimensional center bundle. Thus, it is well known that the non-expansive set for a point $x$ must be contained in a compact subset of a (one-dimensional) center leaf, and so $g$ is entropy-expansive \cite[Proposition 6]{CY05}. We give a quick sketch proof for completeness. 

\begin{lem} \label{centerleaf}
For all $x \in \TT^d$, and $\eps \leq 6\eta$, $\Gamma_{\eps}(x)$ is contained in a compact subset of $W_g^c(x)$ with diameter a uniform multiple of $\eps$.
\end{lem}

\begin{proof}[Sketch proof]
Recall that the foliations $W_g^{cs}$ and $W_g^u$ have a local product structure at scale $\eps$, and that there is also a local product structure within each leaf of $W_g^{cs}$ associated to the foliations $W_g^c$ and $W_g^s$.  In particular, for every $y\in \Gamma_{\eps}(x)$ there are $z_1 \in W_g^{cs}(x) \cap W_g^u(y)$ and $z_2\in W_g^c(x) \cap W_g^s(z_1)$, where all the leaf distances are controlled by a uniform multiple of $\eps$.  Under forward iterates, $z_1$ remains close to $x$, and so if $z_1\neq y$ then uniform forward expansion along leaves of $W_g^u$ implies that for some $n\geq 0$,  $d(g^n(z_1),g^n(y))$ is large enough that $d(g^n(x),g^n(y)) > \eps$.  Thus we must have $z_1 = y$.  A similar argument using backward iterates shows that $z_2 = z_1$.  Thus $y$ is in the local $W_g^c$ leaf of $x$. This shows that $\Gamma_{\eps}(x)$ is contained in a compact subset of $W_g^c(x)$, with diameter a uniform multiple of $\eps$.
\end{proof}

We use this to show there is no tail entropy at scale $6 \eta$, and that Condition \ref{E} from \S\ref{sec:obstr-exp} is satisfied.

\begin{lem} \label{hexp-mane}
The diffeomorphism $g$ satisfies $h_g^\ast(6\eta) =0$.
\end{lem}
\begin{proof}
Given $x \in X$, Lemma \ref{centerleaf} shows that $\Gamma_{6\eta}(x)$ is contained in a compact interval in the center leaf that is bounded in length.  Therefore, $h(\Gamma_{6\eta}(x))=0$ for all $x\in\mathbb{T}^d$ and $h_g^\ast(6\eta)=0$.
\end{proof}

\begin{lem} \label{satisfieshyp}
The diffeomorphism $g$ satisfies Condition \ref{E} from \S\ref{sec:obstr-exp}.
\end{lem}
\begin{proof}
For sufficiently small $\eps>0$,  Lemma \ref{centerleaf} shows that every $x\in \TT^d$ has $\Gamma_\eps(x) \subset W_{\rho/2}^{cs}(x)$.   It follows from Pliss' Lemma~\cite{Pliss} that if $m_k\to \infty$ is such that $\frac 1{m_k} S_{m_k}^{g^{-1}}\chi(x) \geq r$ for every $k$, then for every $r' \in (\gamma, r)$ there exists $m_k'\to\infty$ such that for every $k$ and every $0\leq j\leq m_k'$, we have $\frac 1j S_j^{g^{-1}}\chi(g^{-m_k'+j}x) \geq r'$.  Thus $g^{-m_k'}x$ has the property that
\[
\tfrac 1mS_m^g\chi(g^{-m_k'}x) \geq r' \text{ for all } 0\leq m\leq m_k',
\]
so we can apply Lemma \ref{centrestable} and conclude that
\[
\Gamma_\eps(x) \subset g^{m_k'}(W_{\rho/2}^{cs}(g^{-m_k'}x)) \subset B(x,  \theta_{r'}^{m_k'}\rho/2)
\]
Since $m_k'\to\infty$ and $\theta_{r'}<1$, this implies that $\Gamma_\eps(x) = \{x\}$.
\end{proof}

\subsection{Proof of Theorem \ref{t.mane}} \label{maneprooffinal}

We now complete the proof that if $g\in \UUU_{\rho, \second}$ and $\ph\colon \TT^d\to \RR$ satisfy the hypotheses of Theorem \ref{t.mane}, then the conditions of Theorem \ref{t.generalM} are satisfied, and hence there is a unique equilibrium state for $(\TT^d,g,\ph)$.  We define the decomposition $(\PPP,\GGG, \SSS)$ as in \eqref{eqn:mane-decomp}.  In Lemma \ref{spec}, we showed $\GGG$ has specification at all scales. In Lemma \ref{bowenprop}, we showed $\ph$ has the Bowen property on $\GGG$ at scale $\eps = \frac{\rho}{2\kappa }$. In Theorem \ref{coreestimateallscales}, we showed $P(\PPP, \ph; g)$ admits the upper bound
\[
h_g^*(6\eta) + (1-r) \sup_{B(q, \rho)}\ph + r( \sup_{\mathbb{T}^d}\ph + h + \log L) + H(2r).
\]
By Lemma \ref{hexp-mane},  $h_g^*(6\eta)=0$, and so the hypothesis $\Psi(\rho, r, \ph)< P(\ph;g)$ gives $P(\PPP,\ph)<P(\ph;g)$.
By Theorem \ref{expansivityestimate} and Lemma \ref{satisfieshyp}, $\Pexp(\ph) \leq P(\PPP,\ph)$. Thus, we see that under the hypotheses of Theorem \ref{t.mane}, all the hypotheses of Theorem \ref{t.generalM} are satisfied for the decomposition $(\PPP, \GGG, \SSS)$.

\subsection{H\"older potentials with bounded range} \label{cor.br} We prove the following corollary of Theorem \ref{t.mane}.

\begin{cor}\label{t.manerange}  
Given $g\in \UUU_{\rho,\second}$, suppose that for $L=L(f_A, \eta)$ and $h=\htop(f_A)$, we have
\begin{equation} \label{mane.hestimate}
 r(\log L + h) + H(2r) < h.
\end{equation}
Let $\eta'=C(f_A)d_{C^0}(f_A, g)$ and $V(\varphi)= \Var(\varphi, \eta').$
Then writing $D(r) = h -r(\log L + h) - H(2r) > 0$,
every  H\"older continuous potential $\varphi$ with the bounded range hypothesis 
$\sup \varphi - \inf \varphi +V(\varphi) <D(r)$
has a unique equilibrium state. In particular, \eqref{mane.hestimate} is a criterion for $g$ to have a unique measure of maximal entropy.
\end{cor} 
\begin{proof}
If $\sup \ph - \inf \ph  +V(\varphi)< D(r)$, then
\begin{align*}
 \Psi(\rho, r, \ph) &\leq (1-r) \sup_{B(q, \rho)}\ph + r( \sup_{\TT^d}\ph + h + \log L) +H(2r) +V(\ph) \\
&= (1-r) \sup_{B(q, \rho)}\ph + r(\sup_{\TT^d}\ph) + h +V(\ph) - D(r) \\
&\leq \sup_{\TT^d} \ph + h+V(\ph) - D(r) \\
&< \inf_{\TT^d}\ph + h - V(\ph) 
\leq P(\ph; f_A)- V(\ph)\leq P(\varphi; g).
\end{align*}
The last inequality follows from Lemma \ref{pressuredrop}\ref{Pg-geq}.
Thus Theorem \ref{t.mane} applies.
\end{proof}

\section{Lower bounds on entropy and proof of Theorem \ref{cor1.2}}\label{sec:corollaries}

It is well known  that the unique equilibrium state for a H\"older potential $\ph$ on a uniformly hyperbolic system has positive entropy.   We prove an explicit lower bound on the entropy in terms of $|\ph|_\alpha$, the H\"older semi-norm of $\ph$, for equilibrium states for maps with the specification property.

\begin{thm}\label{thm:pressure-gap}
Let $X$ be a compact metric space and $f\colon X\to X$ a homeomorphism.  Fix $\eps < \frac 16 \diam(X)$ and suppose that $f$ has specification at scale $\eps$ with transition time $\tau$.  Let $\ph\colon X\to\RR$ be a potential satisfying the Bowen property at scale $\eps$ with distortion constant $V$.  Let 
\[
\Delta = \frac{\log(1+e^{-(V+(2\tau+1)(\sup \ph - \inf\ph)})}{2(\tau+1)}.
\]
Then we have
\begin{equation}\label{eqn:gap}
P(\ph) \geq P(\ph,\eps) \geq \Big( \sup_\mu \int\ph\,d\mu\Big) + \Delta.
\end{equation}
In particular, every equilibrium state $\mu$ for $\ph$ has 
$h_\mu(f) \geq \Delta > 0$.
\end{thm}

\begin{proof}
Fix $x\in X$ and 
$n\in \NN$.  Fix $\alpha\in (0,\frac 12]$, let $m_n = \lceil \frac{\alpha n}{2(\tau+1)} \rceil$, and let
\[
\mathcal{I}_n = \{ (k_1, k_2, \ldots, k_{m_n}) \mid 0 < k_1 < k_2 < \cdots < k_{m_n} < n \text{ and } k_i\in 2(\tau+1)\NN\ \forall i\}.
\]
The idea is that for each $\vec k\in \mathcal{I}_n$, we will use the specification property to construct a point $\pi(\vec k) \in X$ whose orbit is away from the orbit of $x$ for a bounded  amount of time around each time $k_i$, and $\epsilon$-shadows the orbit of $x$ at all other times; thus the set of points $\{\pi(\vec k): \vec k \in \mathcal{I}_n\}$ will be $(n,\eps)$-separated on the one hand, and on the other hand each point $\pi(\vec k)$ will have its $n^{th}$ Birkhoff sum close to that of $x$. 

First note that standard estimates for factorials give $\log\binom{n}{\ell} \geq H(\frac\ell{n}) n + o(n)$, and that $m_n / \lfloor \frac n{2(\tau+1)} \rfloor \geq \alpha$, so
\begin{equation}\label{eqn:nmn}
\log\#\mathcal{I}_n
\geq \log{\lfloor \frac{n}{2(\tau+1)}\rfloor \choose m_n} 
\geq \frac{H(\alpha)}{2(\tau+1)} n - o(n).
\end{equation}
Given $k\in \{0, \dots, n-1\}$, let $y_k \in X$ be any point with $d(f^kx,y_k) > 3\eps$. 
Now for every $\vec{k}\in \mathcal{I}_n$, the specification property guarantees the existence of a point $\pi(\vec{k})\in X$ with the property that
\[
\begin{aligned}
\pi(\vec{k}) &\in B_{k_1-\tau}(x,\eps), \\
\qquad f^{k_1}(\pi(\vec{k})) &\in B(y_{k_1},\eps), \\
\qquad f^{k_1+\tau+1}(\pi(\vec{k})) &\in B_{k_2 - k_1 - 2\tau - 1}(f^{k_1 + \tau+1}x),
\end{aligned}
\]
and so on.  Writing $k_0=0$, we see that for any $0\leq i < m_n$ we have
\begin{equation}\label{eqn:pik}
\begin{aligned}
f^{k_i + \tau + 1}(\pi(\vec{k})) &\in B_{k_{i+1} - k_i - 2\tau - 1}(f^{k_i + \tau +1}x), \\
f^{k_{i+1}}(\pi(\vec{k})) &\in B(y_{k_{i+1}},\eps),
\end{aligned}
\end{equation}
and 
we ask that  $f^{k_{m_n}+ \tau+1}(\pi (\vec k)) \in B_{n-k_{m_n}}(f^{k_{m_n}+1+\tau}x)$.  

Write $j_i = k_{i+1} - k_i - 2\tau - 1$ for $i\in \{0, \ldots, m_n-1\}$ and $j_{m_n}=n-k_{m_n}$; then the Bowen property gives
\[
|S_{j_i} \ph(f^{k_i + \tau + 1}x) - S_{j_i} \ph(f^{k_i + \tau + 1} (\pi (\vec{k}))| \leq V
\]
for any $0\leq i \leq m_n$. We control the `excursions' away from $x$ by observing that for any $z,z' \in X$, $|S_{2\tau+1} \ph(z)-S_{2\tau+1} \ph(z')| \leq (2\tau+1)(\sup \ph-\inf \ph)$, and there are $m_n$ such excursions.
We conclude that 
\begin{equation}\label{eqn:Snphi}
|S_n\ph(\pi(\vec{k})) - S_n\ph(x)| \leq (m_n+1)V + m_n(2\tau+1)(\sup \ph- \inf \ph).
\end{equation}

Consider the set $\pi(\mathcal{I}_n) \subset X$.  Given any $\vec{k} \neq \vec{k}' \in \mathcal{I}_n$, let $i$ be minimal such that $k_i \neq k_i'$; then put $j=k_i'$ and observe that
$f^j(\pi(\vec{k})) \in B(y_j, \eps)$ and $f^j(\pi(\vec{k}')) \in B(f^j(x),\eps)$.  Since $d(y_j,f^jx) > 3\eps$ this guarantees that $\pi(\vec{k}') \notin B_n(\pi(\vec{k}),\eps)$, and so $\pi(\mathcal{I}_n)$ is $(n,\eps)$-separated.  Together with \eqref{eqn:Snphi}, this gives
\[
\begin{aligned}
\Lsep_n(\phi,\eps) &\geq \sum_{\vec{k} \in \pi(\mathcal{I}_n)} e^{S_n\ph(\pi(\vec{k}))} \\
&\geq (\#\mathcal{I}_n) \exp\big(S_n\ph(x) - 
(m_n+1)V - m_n(2\tau+1)(\sup\ph-\inf\ph)\big).
\end{aligned}
\]
Using \eqref{eqn:nmn} to bound $\#\mathcal{I}_n$ from below, we can take logs, divide by $n$, and send $n\to\infty$ to get
\[
P(\ph,\eps) \geq \Big(\limsup_{n\to\infty} \frac 1n S_n\ph(x) \Big)
+ \frac 1{2(\tau+1)} \Big(H(\alpha) - \alpha(V+(2\tau+1)(\sup \ph- \inf \ph)) \Big).
\]
Given any ergodic $\mu$, we can take a generic point $x$ for $\mu$ and conclude that the lim sup in the above expression is equal to $\int\ph\,d\mu$.  Thus to bound the difference $P(\ph,\eps) - \int\ph\,d\mu$, we want to choose the value of $\alpha$ that maximizes $H(\alpha) - \alpha Q$, where $Q=V+(2\tau+1)(\sup \ph- \inf \ph)$.

A straightforward differentiation and routine calculation shows that $\frac d{d\alpha} (H(\alpha) - \alpha Q) = 0$ occurs when $\alpha = (1+e^Q)^{-1}$, at which point we have $H(\alpha) - \alpha Q = \log(1+e^{-Q})$, proving Theorem \ref{thm:pressure-gap}.
\end{proof}

\begin{cor}\label{cor:pressure-gap}
Given a topologically mixing Anosov diffeomorphism $f$ on a compact manifold $M$, 
there are $Q,\delta>0$ such that for every H\"older potential $\ph$, we have
\[
P(\ph; f) \geq 
\delta \log(1+e^{-Q|\ph|_\alpha})
 + \sup_\mu \int\ph\,d\mu.
\]
\end{cor}
\begin{proof} 
Every H\"older potential on an Anosov system has the Bowen property with distortion constant given by $Q_1|\ph|_\alpha$; moreover,  $\sup \ph- \inf \ph \leq |\ph|_\alpha (\diam M)^\alpha$. 
Mixing Anosov diffeomorphisms have the specification property;  let $\tau$ be the transition time for a scale at which $f$ has specification, and let $\delta= \frac 1{2(\tau+1)}$.  Then Theorem \ref{thm:pressure-gap} gives
\[
P(\ph; f) \geq \delta \log(1+e^{Q_1 |\ph|_\alpha + (2\tau+1)|\ph|_\alpha(\diam M)^\alpha}).
\]
Putting $Q = Q_1 + (2\tau+1)(\diam M)^\alpha$ gives the result.
\end{proof}

\subsection*{Proof of Theorem \ref{cor1.2}} 

We see from (i) of Lemma \ref{pressuredrop} that there is a constant $K$ (independent of $\rho, \second$) such that for every $g\in \UUU_{\rho,\second}$, 
\[
P(\ph; g) \geq P(\ph; f_A) - K^{\alpha}\rho^\alpha |\ph|_\alpha.
\]
Since $q$ is a fixed point of $f_A$, Corollary \ref{cor:pressure-gap} gives
\[
P(\ph; g) \geq \ph(q) + \delta \log(1+e^{-Q|\ph|_\alpha}) - K^{\alpha} \rho^\alpha |\ph|_\alpha 
\]
for every $g\in \UUU_{\rho,\second}$.  On the other hand, we have 
\begin{align*}
\Psi(\rho,\second,\ph)
&\leq \ph(q) + |\ph|_\alpha\rho^\alpha + r(\sup \ph - \ph(q) + h + \log L) + H(2r) \\
&\leq \ph(q) + |\ph|_\alpha (\rho^\alpha + r(\diam M)^\alpha) + r(h+\log L) + H(2r).
\end{align*}
Thus the following is a sufficient condition to give $\Psi(\rho,r,\ph) < P(\ph; g)$:
\[
(\rho^\alpha(1+K^{\alpha}) + r(\diam M)^\alpha)|\ph|_\alpha  + r(h + \log L) + H(2r)
< \delta \log(1+e^{-Q|\ph|_\alpha})
\]
Let $S_1(\rho,r) = \rho^\alpha(1+K^\alpha) + r(\diam M)^\alpha$ and $S_2(r) = r(h+\log L)+H(2r)$, so the above condition can be rewritten as
\begin{equation}\label{eqn:enough}
S_1(\rho,r)|\ph|_\alpha + S_2(r) < \delta \log(1+e^{-Q|\ph|_\alpha}).
\end{equation}
Given $\rho,r>0$, and $\alpha \in (0,1]$, define $T(\rho,r; \alpha)$ by
\begin{equation}\label{eqn:T}
T(\rho,r;\alpha) = \sup\big\{ T\in \RR : 
S_1(\rho,r) T + S_2(r) < \delta \log (1 + e^{-QT}) \big\}.
\end{equation}
Then for every $\ph$ with $|\ph|_\alpha < T(\rho,r; \alpha)$, condition \eqref{eqn:enough} holds, which gives $\Psi(\rho,r,\ph) < P(\ph; g)$.  This is enough to deduce the first part of Theorem \ref{cor1.2} from Theorem \ref{t.mane}.  For the second part of Theorem \ref{cor1.2}, observe that for every $t>0$, we can choose $\rho,r>0$ sufficiently small that $S_1(\rho,r)t + S_2(r) < \delta \log(1+e^{-Qt})$,
which means that $t < T(\rho,r; \alpha)$ for all sufficiently small $\rho,r$.  In other words, $T(\rho,r; \alpha)\to\infty$ as $\rho,r\to 0$, which completes the proof.

\section{Proof of Theorem \ref{main3}}\label{s.srb}

Given a $C^2$ diffeomorphism $g$ on a $d$-dimensional manifold and $\mu \in \mathcal{M}_e(g)$, let $\lambda_1 \leq \cdots \leq \lambda_d$ be the Lyapunov exponents of $\mu$, and let $\lambda^+(\mu)$ be the sum of the positive Lyapunov exponents.
Following the definition in \cite[Chapter 13]{BP07},  an \emph{SRB measure} for a $C^2$ diffeomorphism is an ergodic invariant measure $\mu$ that is hyperbolic (non-zero Lyapunov exponents) and has absolutely continuous conditional measures on unstable manifolds. The Margulis--Ruelle inequality  \cite[Theorem 10.2.1]{BP07} gives $h_\mu(g) \leq \lambda^+(\mu)$, and it was shown by Ledrappier and Young \cite{LY} that equality holds if and only if $\mu$ has absolutely continuous conditionals on unstable manifolds.  
In particular, for any ergodic invariant measure $\mu$, we have
\begin{equation}\label{eqn:nonpos}
h_\mu(g) - \lambda^+(\mu)\leq 0,
\end{equation}
with equality if and only if $\mu$ is absolutely continuous on unstable manifolds.  Thus an ergodic measure $\mu$ is an SRB measure if and only if it is hyperbolic and equality holds in \eqref{eqn:nonpos}.

Let  $g\in \UUU_{\rho,\second}$ be a $C^2$ diffeomorphism. Since there is a continuous splitting $T \TT^d = E^u \oplus E^{cs}$, the geometric potential 
$\phigeo(x) = -\log\|Dg|_{E^u(x)}\|$ is continuous. Furthermore, $\phigeo$ is H\"older continuous because the map $g$ is $C^2$ and the distribution $E^u$ is H\"older.
The H\"older continuity of $E^u$ follows from the standard argument for Anosov diffeomorphisms.  See for instance \cite[\S6.1]{BrSt}; the argument there extends unproblematically to the case of absolute partial hyperbolicity, which covers our setting.

We build up our proof of Theorem \ref{main3}. Since $\sup \phigeo< 0$, the function $t\mapsto P(t\phigeo)$ is a convex strictly decreasing function, 
so it has a unique root.  
We must show that this root occurs at $t=1$, that we have uniqueness of the equilibrium state for all $t$ in a neighborhood of $[0,1]$, and that the equilibrium state for $\phigeo$, which we denote $\mu_1$, is the unique SRB measure. We assume the hypothesis of Theorem \ref{main3} so that
\begin{equation}\label{eqn:mane-srb-condition2}
r(h+\log L) + H(2r)< \left(\frac{\sup_{x\in \TT^d} \phigeo_{g}(x)}{\inf_{x\in\TT^d} \phigeo_{g}(x)}\right) h, 
\end{equation}
and also that
\begin{equation}\label{eqn:mane-srb-condition3}
r(h+\log L) + H(2r)< - \sup \phigeo_{g}. 
\end{equation}

We recall the following result which was proved as  
Lemma 7.1 of \cite{CFT_BV}.

\begin{lem} \label{lem:Pgeq0}
Let $M$ be a compact Riemannian manifold and let $W$ be a $C^0$ foliation of $M$ with $C^1$ leaves. Suppose there exists $\delta>0$ such that $\sup_{x\in M} m_{W(x)}(W_\delta(x)) < \infty$, where $m_{W(x)}$ denotes volume on the leaf $W(x)$ with the induced metric.  Let $f\colon M\to M$ be a diffeomorphism and let $\psi(x) = -\log|\det Df(x)|_{T_x W(x)}|$.  Then $P(f,\psi)\geq 0$.
\end{lem}

The hypothesis of this lemma is met for a foliation which lies in a cone around a linear foliation, see \cite[\S 7.2]{CFT_BV} for details, so Lemma \ref{lem:Pgeq0} applies to the unstable foliation $W^u$ of the Ma\~n\'e family. We conclude that $P(\ph^u; g) \geq 0$.

To get a unique equilibrium state for $t\phigeo$, it suffices to show that
\[
\Pbad(t) := \Psi(\rho, r, t \phigeo) =  (1-r) \sup_{B(q, \rho)} t\phigeo + r (\sup_{\mathbb{T}^d} t\phigeo + h + \log L ) + H(2r)
\]
satisfies $\Pbad(t) < P(t\phigeo)$ for all $t\in [0,1]$ and then apply Theorem \ref{t.mane}.  Note that  since the equality is strict it will then continue to hold for all $t$ in a neighborhood of $[0,1]$. 


The Ma\~n\'e construction is carried out to leave $E^u$ as unaffected as possible, so we expect that $\sup \phigeo_{f_M}$ and $\inf \phigeo_{f_M}$ are close to $\phigeo_{f_A} \equiv - \log \lambda_u$. Thus, we expect that the $\sup \phigeo_g / \inf \phigeo_g$ term in \eqref{eqn:mane-srb-condition2} can be taken close to $1$. Since making this precise would require a lengthy analysis of the details of the construction with only a small benefit to our estimates, we choose to not pursue this argument. 
For $t\geq0$, we have
\begin{equation} \label{upperline}
\Pbad(t) \leq t(\sup \phigeo) +r (h + \log L ) + H(2r).
\end{equation}
At $t=1$, it is immediate from \eqref{eqn:mane-srb-condition3} that 
\begin{equation} \label{bad0geo}
\Pbad(1)<0 \leq P(\phigeo),
\end{equation}
so $\phigeo$ has a unique equilibrium state. The case $t\in[0,1)$ requires some more analysis. The straight line $l_1(t)$ described by \eqref{upperline}  which bounds $\Pbad(t)$ above has its root at \[
t^\ast = -\frac{r (h + \log L ) + H(2r)}{\sup \phigeo},
\]
and by \eqref{eqn:mane-srb-condition3}, $t^\ast <1$. Thus, for $t\in (t^\ast, 1]$, $\Psi(t)<0\leq P(\phigeo) \leq P(t \phigeo)$.

For $t \in [0, t^\ast]$, the variational principle shows that $P(t\phigeo) \geq h+ t (\inf \phigeo)$. 
Thus we have bounded $P(t\phigeo)$ from below by a straight line $l_2(t)$. In \eqref{upperline},  we bounded $\Pbad(t)$ above by a straight line $l_1(t)$. By \eqref{eqn:mane-srb-condition2}, $l_2(0) > l_1(0)$. The root of $l_2$ is $-h/(\inf \phigeo)$, 
 and the root of $l_1$ is $t^\ast$. Thus by \eqref{eqn:mane-srb-condition2},  $t^\ast<-h/(\inf \phigeo)$ and so $l_2(t^\ast)>  l_1(t^\ast)$. In particular, for $t \in [0, t^\ast]$, $P(t\phigeo) \geq l_2(t)>l_1(t) \geq \Pbad(t)$.
We conclude that $\Pbad(t) < P(t\phigeo)$ for all $t\in [0,1]$, and thus there is a unique equilibrium state by Theorem \ref{t.mane}. 

It remains to show that $P(\phigeo;g)= 0$ and that the unique equilibrium state is in fact the unique SRB measure. 
Let $\mu$ be ergodic, and let $\lambda_1 \leq \lambda_2 \leq \ldots \leq \lambda_d$ be the Lyapunov exponents of $\mu$. Recall that $E^{cs} \oplus E^{u}$ is 
$Dg$-invariant, so for every $\mu$-regular $x$ the Oseledets decomposition is a sub-splitting of $E^{cs} \oplus E^{u}$, and thus $\int \phigeo\,d\mu =- \lambda_d(\mu)$. Thus, 
\begin{equation}\label{eqn:phlambda} 
\int\phigeo\,d\mu \geq -\lambda^+(\mu)
\end{equation}
 and if $\lambda_{d-1}(\mu) <0$ it follows that $\int \phigeo\,d\mu =- \lambda^+(\mu)$. Let $\mathcal{M}_* \subset \mathcal{M}_e(g)$ be the set of ergodic $\mu$ such that $\mu$ is hyperbolic and $\lambda_{d-1}(\mu) <0$.
\begin{lem} \label{lem:keySRBestimate}
If $\mu \in \mathcal{M}_e(g) \setminus \mathcal{M}_*$, then 
\[
h_\mu(g) -\lambda^+(\mu) \leq h_\mu(g) + \int\phigeo\,d\mu \leq \Pbad(1).
\]
\end{lem}
\begin{proof}
If $\mu \in \mathcal{M}_e(g) \setminus \mathcal{M}_*$, then either $\mu$ is not hyperbolic, or $\lambda_{d-1}(\mu)>0$. Then there exists a set $Z\subset M$ with $\mu(Z)=1$ so that for each $z \in Z$, there exists $v \in E^{cs}_z$ with $\lim_{n\to\infty} \tfrac 1n \log \|Dg^n_z(v)\| \geq 0$.

We claim that $z\in Z$ belongs to the set
\[
A^+=\{x : \text{there exists } K(x) \text{ so } \tfrac{1}{n}S^g_{n}\chi(x) < r \text{ for all } n>K(x)\}.
\]
To see this, suppose that $z \notin A^+$. Then there exists $n_k \to \infty$ with $\frac{1}{n_k}S^g_{n_k}\chi(z) \geq r$. By lemma \ref{centrestable}, this gives 
\[
 \|Dg^n_z(v)\| \leq \|Dg^{n_k}|_{E^{cs}(z)}\| < (\theta_r)^{n_k},
\]
and thus $
\lim_{n_k\to\infty} \tfrac 1{n_k} \log \|Dg^{n_k}_z(v)\| \leq \log \theta_r<0$,
which is a contradiction. Thus, $\mu(A^+) =1$. It follows that
\[
h_\mu(g) -\lambda^+(\mu) \leq h_\mu(g) + \int\phigeo\,d\mu 
\leq P(\CCC,\phigeo) \leq \Pbad(1).
\]
where the first inequality uses \eqref{eqn:phlambda}, the second uses Lemma \ref{keystepexpansivityestimate}, and the third uses Theorem \ref{coreestimateallscales}.
\end{proof}

It follows from Lemma \ref{lem:keySRBestimate}, \eqref{bad0geo}, and the Variational Principle that 
\begin{equation}\label{eqn:vp*}
P(\phigeo; g) = \sup \left\{h_\mu(g) + \int\phigeo\,d\mu \, :\,  \mu \in \mathcal{M}_*\right\}.
\end{equation}
Now, for every $\mu\in \mathcal{M}_*$, we have $\int\phigeo\,d\mu = -\lambda^+(\mu)$, and thus
\begin{equation}\label{eqn:free-energy}
h_\mu(g) + \int\phigeo\,d\mu = h_\mu(g) -\lambda^+(\mu)\leq 0
\end{equation}
by \eqref{eqn:nonpos}.  Together with \eqref{eqn:vp*} this gives $P(\phigeo;g) \leq 0$, and we conclude that $P(\phigeo;g)=0$.

It only remains to show that the unique equilibrium state $\mu_1$ is in fact an SRB measure for $g$, and there are no other SRB measures.  Since $\mu_1\in \mathcal{M}_*$,  it is hyperbolic, and since $P(\phigeo;g)=0$, \eqref{eqn:free-energy} gives $h_{\mu_1}(g) - \lambda^+(\mu_1)=0$, so $\mu_1$ is an SRB measure. To see there are no other SRB measures, we observe that if $\nu\neq \mu_1$ is ergodic, then $h_\nu(g) - \lambda^+(\nu) \leq h_\nu(g) + \int\phigeo\,d\nu < P(\phigeo;g)=0$ by the uniqueness of $\mu_1$ as an equilibrium measure. This completes the proof of Theorem \ref{main3}.

\section{Large Deviations and Multifractal Analysis} \label{s.ldp}

\subsection{Large deviations}
The upper level-$2$ large deviations principle is a statement which implies the following estimate on the rate of decay of the measure of points whose Birkhoff sums experience a `large deviation' from the expected value:
\begin{equation}\label{eqn:ldp-1}
\varlimsup_{n\to\infty} \frac 1n \log \mu \left \{x : \left |\frac 1n S_n\psi(x) - \int \psi \, d\mu \right|> \epsilon \right \} \leq -q(\eps),
\end{equation}
where $\eps>0$, $\psi\colon \TT^d\to\RR$ is any continuous function, and $q(\eps) \in[0, \infty] $ is a \emph{rate function}, whose precise value can be formulated precisely in terms of the free energies of a class of measures depending on $\eps$ and $\psi$.
That our equilibrium measures satisfy the upper level-$2$ large deviations principle follows from Theorem 5.5 of \cite{CT4}. That result says that an equilibrium state provided by Theorem \ref{t.generalM} has the upper level-$2$ large deviations principle, and is a consequence of a general large deviations result of Pfister and Sullivan \cite{PfS}, and a weak upper Gibbs property which is satisfied by our equilibrium states. 
The question of lower large deviations bounds for Ma\~n\'e  diffeomorphisms remains open.

\subsection{Multifractal analysis}\label{sec:multifractal}
Let $g$ be a $C^2$ diffeomorphism satisfying the hypotheses of Theorem \ref{main3}. For each $t\in [0,1]$, let $\mu_t$ be the unique equilibrium state for $t\phigeo$ given by Theorem \ref{main3}.  Then, $\mu_0$ is the unique MME and $\mu_1$ is the unique SRB measure.  
It follows from Lemma \ref{hexp-mane} that the entropy map $\mu\mapsto h_g(\mu)$ is upper semicontinuous, hence by Remark 4.3.4 of \cite{Ke98}, uniqueness of the equilibrium state implies that the function $t\mapsto P(t\phigeo)$ is differentiable on $(-\eps,1+\eps)$, with derivative $\chi^+(\mu_t)$, where we write $\chi^+(\mu_t) = \int \phigeo \,d\mu_t$ for the largest Lyapunov exponent of $\mu_t$. This has immediate consequences for multifractal analysis.  Given $\chi\in \RR$, let
\begin{align*}
K_\chi &= \{x\in \TT^d \mid \lim_{n\to\infty} \tfrac 1n \log \|Dg^n|_{E^u(x)}\| = \chi\} \\
&= \{x \in \TT^d \mid \lim_{n\to\infty} \tfrac 1n S_n\phigeo (x) = -\chi\}
\end{align*}
be the set of points whose largest Lyapunov exponent exists and is equal to $\chi$.  The following is a direct consequence of Theorem \ref{main3} and \cite[Corollary 2.9]{C14}.

\begin{thm}\label{thm:multifractal}
Let $g$ and $\mu_t$ be as in Theorem \ref{main3}.  Let $\chi_0 = \chi^+(\mu_0)$ and $\chi_1 = \chi^+(\mu_1)$.  Then for every $\chi\in [\chi_1,\chi_0]$, we have
\begin{align*}
\htop(K_\chi,g) &= \inf\{P(t\ph) + t\chi \mid t\in \RR \} \\
&= \sup \{h_\mu(g) \mid \mu\in \mathcal{M}_f(X), \chi^+(\mu) = \chi\} \\
&= \sup \{h_\mu(g) \mid \mu \in \mathcal{M}_f^e(K_\chi) \},
\end{align*}
where $\htop(K_\chi,g)$ is topological entropy defined as a dimension characteristic  in the sense of Bowen \cite{rB73}. The infimum in the first line is achieved for some $t\in [0,1]$, and for this $t$ we have $\htop(K_\chi,g) = h_{\mu_t}(g)$.
\end{thm}

\subsection*{Acknowledgments} We thank the American Institute of Mathematics, where some of this work was completed as part of a SQuaRE.

\bibliographystyle{plain}
\bibliography{CFT-references-mane}

\end{document}